\documentclass[12pt, reqno]{amsart}
\usepackage{amsmath, amsthm, amscd, amsfonts, amssymb, graphicx, color}
 \usepackage{enumerate}
\usepackage[bookmarksnumbered, colorlinks, plainpages]{hyperref}
\hypersetup{colorlinks=true,linkcolor=red, anchorcolor=green, citecolor=cyan, urlcolor=red, filecolor=magenta, pdftoolbar=true}

\newtheorem{theorem}{Theorem}[section]
\newtheorem{lemma}[theorem]{Lemma}
\newtheorem{proposition}[theorem]{Proposition}

\newtheorem{definition}[theorem]{Definition}
\newtheorem{example}[theorem]{Example}

\theoremstyle{remark}
\newtheorem{remark}[theorem]{\bf{Remark}}
\numberwithin{equation}{section}
\allowdisplaybreaks
\begin{document}

\title [A note on weighted composition operators on Dirichlet space]
{\Small{A note on weighted composition operators on Dirichlet space}}
	\author[S. Halder, S. Mukherjee, R. Birbonshi ]{Subhadip Halder, Sweta Mukherjee, Riddhick Birbonshi}

\address[Halder]{Department of Mathematics, Jadavpur University, Kolkata 700032, West Bengal, India}
\email{subhadiphalderju@gmail.com}

\address[Mukherjee] {Department of Mathematics, Jadavpur University, Kolkata 700032, West Bengal, India}
\email{sweta.bankati@gmail.com}
\address[Birbonshi] {Department of Mathematics, Jadavpur University, Kolkata 700032, West Bengal, India}
\email{riddhick.math@gmail.com}

\subjclass[2020]{47A12, 47B32, 47B33.}

\keywords{Numerical range, Compact operators, Dirichlet space, Weighted composition operators.}

\maketitle

\begin{abstract}
    In this paper, we provide some sufficient conditions for the compactness of weighted composition operators on Dirichlet space. Furthermore, we characterize the numerical range of certain classes of weighted composition operators on Dirichlet space and establish the criteria that guarantee the inclusion of zero within the numerical range. Finally, several classes of weighted composition operators are introduced, whose numerical range contains either a circular disk with a given center and radius or an elliptical disk with specified foci and axis lengths.
\end{abstract}

\section{introduction}
Let $\mathbb{D}=\{z \in \mathbb{C} : |z|<1\}$ be the open unit disk, and let $H(\mathbb{D})$ denote the space of all holomorphic  functions on $\mathbb{D}.$  Let $H^{\infty}$ be the space of bounded, holomorphic  functions on $\mathbb{D}$. 
For a function $f\in H(\mathbb{D})$, its Dirichlet integral is defined by
$$\mathcal{D}(f)=\frac{1}{\pi}\int_{\mathbb{D}}|f^{\prime}(z)|^2dA(z),$$ where $dA$ denotes the area Lebesgue measure on $\mathbb{D}.$
The Dirichlet space $\mathcal{D}$ is the collection of all $f\in H(\mathbb{D})$ for which $\mathcal{D}(f)$ is finite.
Clearly, $\mathcal{D}$ contains all the polynomials, more generally, all functions $f$ holomorphic on $\mathbb{D}$ such that $f^{\prime}$ is bounded on $\mathbb{D}$. If $f$ belongs to the Dirichlet space $\mathcal{D}$, then it immediately follows that its derivative $f'$ is in the Bergman space $A^2$. Moreover, if a function $f\in\mathcal{D}$, say $f(z)=\sum \limits_{k\geq0}a_kz^k$ then the quantity $$\mathcal{D}(f)=\sum\limits_{k\geq1}k|a_k|^2<\infty.$$\\
Consequently, every function in $\mathcal{D}$ is also an element of the Hardy space $H^2(\mathbb{D})$ as well as
the Bergman space $A^2$. 
For $f,g \in\mathcal{D}$, the inner product on $\mathcal{D}$ is given by
$$ \langle f,g \rangle_{\mathcal{D}}=\langle f,g \rangle_{H^2(\mathbb{D})}+\frac{1}{\pi}\int_{\mathbb{D}}f^{\prime}(z)\overline{g^{\prime}(z)}dA(z),$$
where $\langle.,.\rangle_{H^2(\mathbb{D})}$ is the usual inner product on Hardy space $H^2(\mathbb{D})$. The corresponding norm is therefore 
$$\|f\|^2_{\mathcal{D}}=\|f\|^2_{H^2(\mathbb{D})}+\frac{1}{\pi}\int_{\mathbb{D}}|f^{\prime}(z)|^2dA(z).$$
For $f(z)=\sum_{n=0}^{\infty}\hat{f}_nz^n$ and $g(z)=\sum_{n=0}^{\infty}\hat{g}_nz^n$ and $f,g\in\mathcal{D}$, inner product can also be expressed by \begin{align*}
\langle f,g \rangle_{\mathcal{D}}=\sum_{n=0}^{\infty}(n+1)\hat{f}_n\overline{\hat{g}}_n.
\end{align*}
Under this inner product, the collection of functions $e_n(z)=\frac{z^n}{\sqrt{n+1}},n\geq0$
 form an orthonormal basis for the space $\mathcal{D}$.\\
 For any  $w\in\mathbb{D}\setminus\{0\}$, define  $k_w(z)=\frac{1}{z\overline{w}}log\left(\frac{1}{1-z\overline{w}}\right),$   $z\in\mathbb{D}$ and set $k_0\equiv1$, where $log$ refers to principal branch of logarithm.
 Then $k_w,k_0\in\mathcal{D}$ and these functions are called the reproducing kernels of $\mathcal{D}$. Moreover, for each $w\in\mathbb{D}\setminus\{0\},$
 $$\|k_w\|_\mathcal{D}^2=\frac{1}{|w|^2}log\left(\frac{1}{1-|w|^2}\right).$$
 Another frequently used norm in the Dirichlet space $\mathcal{D}$  is  $$\|f\|^2=|f(0)|^2 +\frac{1}{\pi}\int_{\mathbb{D}}|f^{\prime}(z)|^2dA(z),$$ which is equivalent to $\|\cdot\|_{\mathcal{D}}.$ Unless otherwise specified, in this article the Dirichlet space is assumed to be equipped with the inner product $\langle.,.\rangle_{\mathcal{D}}$ and the corresponding norm is $\|\cdot\|_{\mathcal{D}}.$\\
 A positive borel measure $\mu$ on $\mathbb{D}$ is called Carleson measure for $\mathcal{D}$ if there exists a positive constant $c$ such that
\begin{align*}
    \int_{\mathbb{D}}|f|^2 d\mu&\leq c\|f\|^2_{\mathcal{D}},
\end{align*}
for all $f\in\mathcal{D}.$
 A multiplier of $\mathcal{D}$ is a function $h:\mathbb{D} \to \mathbb{C}$   with the property that $hf \in \mathcal{D}$ for every $f \in \mathcal{D}$. The collection of all such multipliers forms an algebra known as the multiplier algebra of  $\mathcal{D}$ which we denote by $M(\mathcal{D})$. A function $h$ belongs to  $M(\mathcal{D})$ if and only if $h$ is bounded and $|h^{\prime}|^2dA$ is a Carleson measure for $\mathcal{D}$. It is well established that multiplication operator $M_{\psi}$ is bounded for any   $\psi\in M(\mathcal{D})$. For more information on the Dirichlet space, we refer to \cite{MASHREGHI_BOOK}.

Let $\phi$ be an analytic self map of the open unit disk $\mathbb{D}$ and $\psi$ be another analytic function on $\mathbb{D}$. The weighted composition operator $C_{\psi,\phi}$ on the Dirichlet space $\mathcal{D}$ is defined by $(C_{\psi,\phi}f)(z)=\psi(z)f(\phi(z))$, for  $f\in\mathcal{D}$, $z\in\mathbb{D}$. When $\psi\equiv1$,  this operator reduces to usual composition operator.
The boundedness and compactness of  composition operators on the Dirichlet  space have been studied extensively (see \cite{Daniel_L_Dirichlet,Lefèvre_ Pascal_ et al, Mirzakarimi_Seddighi}).  In \cite{Kriete_ Moorhouse}, Kriete and Moorhouse examined the compactness of weighted composition operators for a bounded weight $\psi$. Later Gunatillake \cite{GJS_AMS_2008} showed  that  $C_{\psi,\phi}$ can be compact on $H^2(\mathbb{D})$ even if  $\psi$ is not bounded. Interestingly \cite{GJS_AMS_2008} demonstrate that  $C_{\psi,\phi}$ can be compact even when $C_\phi$ is not compact. In  \cite{LO_2022}, Lo and Loh  also gave  some sufficient conditions for the compactness of composition operators on weighted Bergman spaces. In this work, we aim to determine the conditions under which  $C_{\psi,\phi}$ is compact on $\mathcal{D}.$

 Let $T$ be a bounded linear operator on a Hilbert space $\mathcal{H}$.  Numerical range of $T$ is defined as $$W(T)=\left\{\langle Tf,f \rangle : f \in \mathcal{H}, \|f\|=1 \right\}.$$ For any set $S\subseteq\mathbb{C}$, we denote the convex hull of $S$ by $S^{\wedge}$, and $\mathbb{B}\left(\mathcal{H}\right)$ represents  the space of all bounded linear operators on $\mathcal{H}$. It is well known fact that if $T\in\mathbb{B}\left(\mathcal{H}\right)$, the numerical range $W(T)$ is convex and $\overline{W(T)}$ contains the spectrum of $T$. For more information about numerical range we refer to  \cite{wu_gau_book,GR_BOOK}. 
 The numerical range of composition operators on the Hardy space has been widely examined in the literature (see \cite{BS_JMAA_2000,BS_IEOT_2002,M_LAA_2001}). In \cite{GJS_JMAA_2014}, Gunatillake examined several properties of the numerical range of weighted composition operators on the Hardy space. Later in \cite{N_range}, several interesting results of the numerical ranges of Toeplitz and weighted composition operators on weighted Bergman spaces has been addressed.

In this article, we examine the compactness and numerical range of weighted composition operators on the Dirichlet space $\mathcal{D}$. Section 2 presents some sufficient conditions under which a weighted composition operator is compact on 
$\mathcal{D}$. Section 3 identifies the numerical ranges of certain weighted composition operators and establishes criteria ensuring that zero lies within the numerical range. Finally, in section 4, we introduce several classes of weighted composition operators whose numerical ranges include either a circular disk with specified centre and radius or an elliptical disk with given foci and lengths of axes.

\section{Compact weighted Composition operators}\label{S1}
In \cite{GJS_AMS_2008,LO_2022,Lo_H^p}, several sufficient conditions for the compactness of weighted composition operators on Hardy and Bergman Spaces are provided. Here, we turn our attention to the problem of identifying when these operators are compact on the Dirichlet space.\\
 We begin with a useful result from \cite[Theorem 2.3]{Gall_2015}.
\begin{lemma}\label{L0}
    Let $\phi:\mathbb{D}\rightarrow\mathbb{D}$ be analytic map and $M(\phi)=\{\psi\in\mathcal{D}:C_{\psi,\phi} $ is bounded on $\mathcal{D} \}$. Then $M(\phi)=\mathcal{D}$ if and only if 
\begin{eqnarray*}
    (i)\,\|\phi\|_\infty<1 \,\,\,\,  \mbox{and}\,\,\,\,  (ii) \,\phi\in M(\mathcal{D}).
\end{eqnarray*}
   
\end{lemma}
In \cite[Theorem 2]{GJS_AMS_2008}, Gunatillake provided a sufficient condition for compactness of $C_{\psi,\phi}$ on $H^2({\mathbb{D}}).$  Here, our goal is to investigate whether an analogous result holds for the Dirichlet space. To proceed, we first present the following lemmas.
\begin{lemma}
    Let $\phi\in M(\mathcal{D})$ and $\overline{\phi(\mathbb{D})}\subseteq\mathbb{D}.$ Then for any $\psi\in\mathcal{D},$ $C_{\psi,\phi}$ is bounded in $\mathcal{D}.$
\end{lemma}
\begin{proof}
    Since we have $\phi\in M(\mathcal{D})$ and $\overline{\phi(\mathbb{D})}\subseteq\mathbb{D},$ so from Lemma \ref{L0}, $M(\phi)=\mathcal{D}.$ Hence, for any $\psi\in\mathcal{D},$ $C_{\psi,\phi}$ is bounded on $\mathcal{D}.$
\end{proof}

\begin{lemma}\label{L1}\cite{LIU}
Suppose $C_{\psi,\phi}$ is a bounded operator on $\mathcal{D}$. Then $C_{\psi,\phi}$ is compact if and only if whenever $\{f_{n}\}$ is a bounded sequence in $\mathcal{D}$ and $\{f_{n}\}$ uniformly converges to 0 on  compact subsets of $\mathbb{D}$, $\{C_{\psi,\phi}(f_{n})\}$ converges to 0 in $\mathcal{D}$.
\end{lemma}
\begin{lemma}\label{LL1}\cite{Stein_book}
Suppose a sequence of holomorphic functions $\{f_n\}$ on $\mathbb{D}$ converges uniformly to 0 on all compact subsets of $\mathbb{D}.$ Then $\{f^{\prime}_n\}$ converges uniformly to 0 on all compact subsets of $\mathbb{D}$.
 \end{lemma}
\begin{theorem}\label{Th1}
    
  Suppose $\overline{\phi(\mathbb{D})}$ is contained in the open unit disk $\mathbb{D}$ and $|\phi^{\prime}|^2dA$ is Carleson measure in $\mathcal{D}$. Then for any $\psi$ in $\mathcal{D}$, the weighted composition operator $C_{\psi,\phi}$ is compact in $\mathcal{D}$.
  \end{theorem}
\begin{proof}
    
 Since $\overline{\phi(\mathbb{D})}\subseteq \mathbb{D}$, there exists $M$ with $0<M<1$ such that $|\phi(z)|\leq M, \forall z\in\mathbb{D}$.
Clearly, $\|\phi\|_\infty<1$ and since $|\phi^{\prime}|^2dA$ is Carleson measure, so by \cite[Theorem 5.1.7]{MASHREGHI_BOOK}, $\phi\in M(\mathcal{D})$. Then by Lemma \ref{L0}, $C_{\psi,\phi}$ is bounded on $\mathcal{D}.$\\
Let $\{f_n\}$ be a bounded sequence in $\mathcal{D}$, and $f_n \rightarrow 0$ uniformly on compact subsets of $\mathbb{D}$.
  Therefore, $f_n\rightarrow 0$ uniformly on $\overline{S(0,M)}$, where  $S(0,M)$ denotes the open ball centered at 0 with radius $M$. Consequently, there exists $N\in\mathbb{N}$ for all $n>{N}$ 
  $|f_n(\phi(z))|<\epsilon$ and $|f^{\prime}_n(\phi(z))|<\epsilon$, for every  $z\in\mathbb{D}$ by Lemma \ref{LL1}. Hence, for all $n>N$,
\begin{align*}
    \|C_{\psi,\phi}(f_n)\|^2_{\mathcal{D}}&=\|C_{\psi,\phi}(f_n)\|^2_{{H^{2}({\mathbb{D})}}}+\frac{1}{\pi}\int_{\mathbb{D}}|(C_{\psi,\phi}(f_n))^\prime(z)|^2dA(z)\\
    &=\|C_{\psi,\phi}(f_n)\|^2_{{H^{2}({\mathbb{D})}}}+\frac{1}{\pi}\int_{\mathbb{D}}|(\psi(f_n\circ\phi))^{\prime}(z)|^2dA(z)\\
    &=\|C_{\psi,\phi}(f_n)\|^2_{{H^{2}({\mathbb{D})}}}+\frac{1}{\pi}\int_{\mathbb{D}}|\psi^{\prime}(z)(f_n\circ\phi)(z)+\psi(z) f_n^{\prime}(\phi(z))\phi^{\prime}(z)|^2 dA(z)\\
    &\leq\|C_{\psi,\phi}(f_n)\|^2_{{H^{2}({\mathbb{D})}}}+\frac{2}{\pi}\int_{\mathbb{D}}\left(|\psi^{\prime}(z)(f_n\circ\phi)(z)|^2+|\psi(z) f_n^{\prime}(\phi(z))\phi^{\prime}(z)|^2\right) dA(z). 
\end{align*}
    
 Since $\psi\in\mathcal{D}\subseteq{H^2(\mathbb{D})},$ $\{f_n\}$ is bounded on $H^2(\mathbb{D})$, then from \cite[Lemma 1, Theorem 2]{GJS_AMS_2008}, $C_{\psi,\phi}(f_n)\rightarrow 0$  on $H^2(\mathcal{D}).$
 Again,
\begin{align*}
    \int_{\mathbb{D}}
    |\psi^{\prime}(z)(f_n\circ\phi)(z)|^2dA(z)&\leq\epsilon^2\int_{\mathbb{D}}|\psi^{\prime}(z)|^2dA(z) \\
    &=\epsilon^2\|\psi^{\prime}\|_{A^2} 
\end{align*}
and
\begin{align*}
    \int_{\mathbb{D}}|\psi(z) f_n^{\prime}(\phi(z))\phi^{\prime}(z)|^2 dA(z)&\leq\epsilon^2\int_{\mathbb{D}}|\psi(z)\phi^{\prime}(z)|^2dA(z)\\
    &=\epsilon^2\int_{\mathbb{D}}|\psi(z)|^2|\phi^{\prime}(z)|^2dA(z).
\end{align*}
Since $|\phi^{\prime}|^2dA$ is Carleson measure in $\mathcal{D},$ then we get
\begin{align*}
    \int_{\mathbb{D}}|\psi(z) f_n^{\prime}(\phi)(z)\phi^{\prime}(z)|^2 dA(z)&\leq\epsilon^2c\|\psi\|_{\mathcal{D}},
\end{align*}
for some $c>0$. Therefore, $\|C_{\psi,\phi}(f_n)\|_{\mathcal{D}}$ tends to 0. 
Hence, from Lemma \ref{L1}, $C_{\psi,\phi}$ is compact in $\mathcal{D}$.
\end{proof}
 We now provide an example demonstrating Theorem \ref{Th1} when $\psi$ is unbounded on $\mathbb{D}.$ 
\begin{example}
    Let $\phi(z)=\frac{z^2}{2}$ and $\psi(z)=\sum\limits_{k\geq2}\frac{z^k}{k\log k}, \,z\in\mathbb{D}.$ Then from \cite[Page 4]{MASHREGHI_BOOK}, $\psi\in\mathcal{D}$ and $\psi$ is unbounded on $\mathbb{D}$. Since $M(\mathcal{D})$ includes all polynomials, it follows from \cite[Theorem 5.1.7]{MASHREGHI_BOOK} that the measure $|\phi^{\prime}|^2dA$ is a Carleson measure. Therefore the functions $\psi$ and $\phi$ satisfy the conditions of Theorem \ref{Th1}, and it follows that the operator $C_{\psi,\phi}$ is compact on $\mathcal{D}.$
\end{example}

In \cite[Theorem 2.8]{LO_2022}, Lo and Loh established a condition ensuring the compactness of $C_{\psi,\phi}$ on weighted Bergman spaces. Here, we examine whether the same conditions guarantee compactness in the Dirichlet space. The answer turns out to be negative. To demonstrate this, we construct an example of a weighted composition operator that satisfies all the assumptions of \cite[Theorem 2.8]{LO_2022} for $\alpha=0$. Consequently, the operator is compact on Bergman space but fails to be compact in the Dirichlet space. To establish this example, first we prove the following result. 
\begin{theorem}\label{Th0}
    If $C_{\psi,\phi}$ is compact on $\mathcal{D}$, then $\lim\limits_{|z|\to 1^{-}}\tilde{f}(z)=0$, where $$\tilde{f}(z)=\begin{cases}
    |\psi(z)|\frac{|z|}{|\phi(z)|}\sqrt{\frac{log\left(\frac{1}{1-|\phi(z)|^2}\right)}{log\left(\frac{1}{1-|z|^2}\right)}}, &  \phi(z)\neq 0,z\neq0\\
   \frac{ |\psi(0)|}{|\phi(0)|}\sqrt{log\left(\frac{1}{1-|\phi(0)|^2}\right)},& \phi(z)\neq0,z=0\\
    |\psi(0)|,& \phi(z)=0,z=0\\
    \frac{|\psi(z)\|z|}{\sqrt{log\left(\frac{1}{1-|z|^2}\right)}},& \phi(z)=0, z\neq 0
    \end{cases}.$$
\end{theorem}
\begin{proof}
    If $a$ is an arbitrary point in $\mathbb{D}$, then $\langle f,k_a\rangle=f(a)$ for all $f$ in $\mathcal{D}$, where the reproducing kernel at $z=a$ is given by $$k_a(z)=\frac{1}{\overline{a}z}log\left(\frac{1}{1-\overline{a}z}\right),a\neq 0,\forall z\in\mathbb{D}$$ and $$k_0\equiv1.$$\\
    The norm of the kernel is $\|k_a\|^2_{\mathcal{D}}=\frac{1}{|a|^2}log\left(\frac{1}{1-|a|^2}\right)$.
    Further, for every $g\in\mathcal{D},$ $\langle(C_{\psi,\phi})^*k_a,g\rangle=\langle\overline{\psi(a)}k_\phi(a),g\rangle.$ 
    This implies $(C_{\psi,\phi})^*k_a=\overline{\psi(a)}k_{\phi(a)}.$ Therefore,
    \begin{align*}
    \|(C_{\psi,\phi})^*K_z\|_{\mathcal{D}}&=\left\|(C_{\psi,\phi})^*\frac{k_z}{\|k_z\|_{\mathcal{D}}}\right\|_{\mathcal{D}}\\&=\frac{|\psi(z)|\|k_{\phi(z)}\|_{\mathcal{D}}}{\|k_z\|_{\mathcal{D}}}\\&=|\psi(z)|\frac{|z|}{|\phi(z)|}\sqrt{\frac{log\left(\frac{1}{1-|\phi(z)|^2}\right)}{log\left(\frac{1}{1-|z|^2}\right)}},
    \end{align*}
    where $|\phi(z)|\neq 0,\, z\neq0$ and $K_z$ denotes the normalized reproducing kernel on $\mathcal{D}$ at $z$.\\
    When $|\phi(z)|\neq 0\, z=0;$ we have $\|(C_{\psi,\phi})^*K_z\|_{\mathcal{D}}=
   \frac{ |\psi(0)|}{|\phi(0)|}\sqrt{log\left(\frac{1}{1-|\phi(0)|^2}\right)}.$ \\
   For $|\phi(z)|= 0$ and $z=0;$ we get $\|(C_{\psi,\phi})^*K_z\|_{\mathcal{D}}=|\psi(0)|.$\\
   Similarly
    when $\phi(z)=0$ and $z\neq 0;$ we have $\|(C_{\psi,\phi})^*K_z\|_{\mathcal{D}}=
    \frac{|\psi(z)\|z|}{\sqrt{log\left(\frac{1}{1-|z|^2}\right)}}.$\\
    Since $C_{\psi,\phi}$ is compact, its adjoint $(C_{\psi,\phi})^*$ is compact.  Moreover, by \cite[Theorem 2.17]{Cowen_book}, $K_z$ converges weakly to 0 in $\mathcal{D}$ as $|z|\rightarrow 1^{-}$. Hence, the desired condition follows.
\end{proof}

\begin{example}
    Let $\psi(z)=(1-\sqrt{1-z})^2$ and $\phi(z)=1-\sqrt{1-z}.$ It is straightforward to verify that these functions satisfy condition (i) of \cite[Theorem 2.8]{LO_2022}. Furthermore, since $\lim\limits_{|z|\rightarrow1^{-}} \frac{|1-\sqrt{1-z}|(1-|z|^2)}{|\sqrt{1-z}|}=0$, condition (ii) of \cite[Theorem 2.8]{LO_2022} is also satisfied. Clearly $\psi\in A^2$ and $\phi$ has no angular derivative on the boundary of the unit disk, \cite[Theorem 3.22]{Cowen_book} implies that $C_{\phi}$ is compact on $A^2$. Because $\psi$ is bounded on $\mathbb{D},$ \cite[Theorem 2.6]{LO_2022} ensures that $C_{\psi,\phi}$ is compact on $A^2.$ Applying \cite[Theorem 2.4]{LO_2022}, we conclude that the condition (iii) of \cite[Theorem 2.8]{LO_2022} also holds and hence, $C_{\psi,\phi}$ is compact on the Bergman space. Next, we consider the following limit
    \begin{align*}
  \lim\limits_{r\to 1^{-}}\frac{|\psi(r)|r}{\phi(r)}\sqrt{\frac{log(\frac{1}{1-(\phi(r))^2})}{log(\frac{1}{1-r^2})}}
  &=\lim\limits_{r\to 1^{-}}(1-\sqrt{1-r})r\sqrt{\frac{log(\frac{1}{1-(\phi(r))^2})}{log(\frac{1}{1-r^2})}}\\
  &=\frac{1}{\sqrt{2}}\neq 0.
  \end{align*}
  Therefore, by Theorem \ref{Th0}, $C_{\psi,\phi}$ is not compact on $\mathcal{D}$.\\ From this, we can conclude that any operator  $C_{\psi,\phi}$ that is compact on $A^2$ need not be compact on $\mathcal{D}$.
    \end{example}
 From the above example, we observe that the conditions stated in \cite[Theorem 2.8]{LO_2022}, are not sufficient to ensure the compactness of $C_{\psi,\phi}$ on the Dirichlet space $\mathcal{D}.$ Therefore, our objective is to establish some additional conditions, along with those of \cite[Theorem 2.8]{LO_2022}, that will guarantee the compactness of the operator $C_{\psi,\phi}$ on $\mathcal{D}.$ To proceed, we first recall the following Lemma from \cite{LO_2022}, for the case $\alpha=0.$

\begin{lemma}\label{L2}
 If $f$ is in the Bergman space $A^2$, then 
$$\frac{2}{3}\|f\|^2_{A^2}\leq|f(0)|^2+\int_{\mathbb{D}}|f^{\prime}(z)|^2(1-|z|^2)^2dA(z)\leq 2\|f\|^2_{A^2}.$$

\end{lemma}
Now, we have the following result.
\begin{theorem}\label{Th3}
Let $C_{\psi,\phi}$ be a weighted composition operator on the Dirichlet space $\mathcal{D}$. Assume that  $\phi^{\prime},\phi^{\prime \prime} $ and 
 $\psi,\psi^{\prime}$ are all bounded functions. If the following conditions hold: \begin{enumerate}[\upshape (i)]
    \item $\phi$ is univalent on $\mathbb{D}$
    \item  $\lim\limits_{|z|\rightarrow 1^-}\frac{|\psi(z)|(1-|z|^2)}{1-|\phi(z)|^2}=0,$\label{bbb1}
    \item $\lim\limits_{|z|\rightarrow 1^-}|\psi^{\prime \prime}(z)|(1-|z|^2)=0,$\label{bbb2}
\end{enumerate}

then the operator $C_{\psi,\phi}$ is compact on $\mathcal{D}$.

\end{theorem}
\begin{proof}
Since $\psi^{\prime}$ is bounded,  $$\lim\limits_{|z|\to 1^{-}}|\psi^{\prime}(z)|(1-|z|^2)=0.$$
Let $\epsilon>0$.  Using conditions (\ref{bbb1}),(\ref{bbb2}) along with the boundedness of $\psi^{\prime},$ there exists a constant $r$ with $\frac{1}{2}<r<1$ such that for all $z$ satisfying $r<|z|<1,$ the following inequalities hold: $$|\psi^{\prime}(z)|(1-|z|^2)<\epsilon,$$
$${|\psi(z)|(1-|z|^2)}<\epsilon({1-|\phi(z)|^2}),$$
$$|\psi^{\prime \prime}(z)|(1-|z|^2)<\epsilon.$$ \\
 Moreover, there exists a constant $M>0$ such that the functions $\phi,\phi^{\prime},\phi^{\prime\prime}, \psi$ and $\psi^{\prime}$ are all bounded by $M.$ By \cite[Theorem 4.7]{Lo_H^p}, the operator $C_{\psi,\phi}$ is bounded on $H^2(\mathbb{D})$. Hence, the boundedness of $C_{\psi,\phi}$ on $\mathcal{D}$ follows directly from the boundedness of $\psi,\psi^{\prime}$ and $\phi^{\prime}$, together with the fact that $C_{\psi,\phi}$ is bounded on $H^2({\mathbb{D})}$.\\
Let $\{f_n\}$ be a bounded sequence in $\mathcal{D}$ that converges uniformly to 0 on every compact subset of $\mathbb{D}$. 
Denote by $S(0,r)$ the open disk centred at the origin with radius r, and let $\overline{S(0,r)}=S_1$. Then  
we have,
$$\|C_{\psi,\phi}(f_n)\|^2_{\mathcal{D}}=\|C_{\psi,\phi(f_n)}\|^2_{{H^{2}({\mathbb{D})}}}+\frac{1}{\pi}\int_{\mathbb{D}}|(C_{\psi,\phi}(f_n))^{\prime}(z)|^2dA(z).$$
Since $\{f_n\}$ is a bounded sequence in $\mathcal{D}$, it is also bounded in $H^2(\mathbb{D}).$ It follows directly from  \cite[Theorem 4.7]{Lo_H^p} that the operator $C_{\psi,\phi}$ is compact in $H^2(\mathbb{D}).$ Therefore by \cite[Lemma 2.1]{Lo_H^p}, we have   $\|C_{\psi,\phi}(f_n)\|_{H^2(\mathbb{D})}\rightarrow 0$. Now,
 \begin{align*}
    \int_{\mathbb{D}}|(C_{\psi,\phi}(f_n))^{\prime}(z)|^2dA(z)&= \int_{\mathbb{D}}|(\psi(f_n\circ\phi))^{\prime}(z)|^2dA(z)\\&= \int_{\mathbb{D}}|\psi^{\prime}(z)(f_n\circ\phi)(z)+\psi(z) f^{\prime}_n(\phi(z))\phi^{\prime}(z)|^2 dA(z).
\end{align*}

We have, $(C_{\psi,\phi}(f_n))^{\prime}\in A^2$, therefore, by Lemma \ref{L2} it follows that
\begin{eqnarray}
&&\frac{2}{3}\|(\psi(f_n\circ\phi))^{\prime}\|^2_{A^2}\nonumber\\
&\leq& |\psi^{\prime}(0)f_n(\phi(0))+\psi(0)\phi^{\prime}(0)f^{\prime}_n(\phi(0))|^2\nonumber\\&&+\int_\mathbb{D}|(\psi^{\prime}(f_n\circ\phi))^{\prime}(z)+(\psi (f^{\prime}_n\circ\phi)\phi^{\prime})^{\prime}(z)|^2(1-|z|^2)^2dA(z)\label{bandor}\\ &\leq& 2\|(\psi(f_n\circ\phi))^{\prime}\|^2_{A^2}.\nonumber
\end{eqnarray}
Since $\{0\}$ and $\{\phi(0)\}$ are both compact sets, and $f^{\prime}_n(\phi(0))\rightarrow0$ by Lemma \ref{LL1}, it follows that $|\psi^{\prime}(0)f_n(\phi(0))+\psi(0)\phi^{\prime}(0)f^{\prime}_n(\phi(0))|\rightarrow0.$\\ Again,
\begin{eqnarray*}
    &&\int_\mathbb{D}|(\psi^{\prime}(f_n\circ\phi))^{\prime}(z)+(\psi (f^{\prime}_n\circ\phi)\phi^{\prime})^{\prime}(z)|^2(1-|z|^2)^2dA(z)\\
    &=&\int_{\mathbb{D}\setminus S_1}|(\psi^{\prime}(f_n\circ\phi))^{\prime}(z)+(\psi (f^{\prime}_n\circ\phi)\phi^{\prime})^{\prime}(z)|^2(1-|z|^2)^2dA(z)\\
    &&+\int_{S_1}|(\psi^{\prime}(f_n\circ\phi))^{\prime}(z)+(\psi (f^{\prime}_n\circ\phi)\phi^{\prime})^{\prime}(z)|^2(1-|z|^2)^2dA(z)\\
    &\leq& 8\Bigg(\int_{\mathbb{D}\setminus S_1}|\psi^{\prime}(z)(f_n\circ\phi)(z)|^2(1-|z|^2)^2dA(z)+\\&&2\int_{\mathbb{D}\setminus S_1}|\psi^{\prime}(z)(f^{\prime}_n\circ\phi)(z)\phi^{\prime}(z)|^2(1-|z|^2)^2dA(z)+
    \\&&\int_{\mathbb{D}\setminus S_1}|\psi(z)(f^{\prime\prime}_n\circ\phi)(z)(\phi^{\prime})^2(z)|^2(1-|z|^2)^2dA(z)+\\&&\int_{\mathbb{D}\setminus S_1}|\psi(z)\phi^{\prime\prime}(z)(f^{\prime}_n\circ\phi)(z)|^2(1-|z|^2)^2dA(z)\Bigg)+\\&&\int_{S_1}|(\psi^{\prime}(f_n\circ\phi))^{\prime}+(\psi (f^{\prime}_n\circ\phi)\phi^{\prime})^{\prime}|^2(1-|z|^2)^2dA(z).
 \end{eqnarray*}
Since $S_1$ is compact, the sequences $\{f_n\}$,$\{f_n^{\prime}\}$ and $\{f_n^{\prime\prime}\}$ all converges uniformly to 0 on $S_1$ by Lemma \ref{LL1}.
 
Hence, \\  $$\int_{S_1}|(\psi^{\prime}(z)(f_n\circ\phi)(z))^{\prime}+(\psi(z) (f^{\prime}_n\circ\phi)(z)\phi^{\prime}(z))^{\prime}|^2(1-|z|^2)^2dA(z)\rightarrow 0.$$ \\ 
Now, for the region $\mathbb{D}\setminus S_1$, we have

\begin{align*}
    \int_{\mathbb{D}\setminus S_1}|\psi^{\prime}(z)|^2|(f^{\prime}_n\circ\phi)(z)|^2|\phi^{\prime}(z)|^2(1-|z|^2)^2dA(z)&\leq
    \epsilon^2\int_{\mathbb{D}\setminus S_1}|\phi^{\prime}(z)|^2|(f^{\prime}_n\circ\phi)(z)|^2dA(z)\\&\leq
    M\epsilon^2\int_{\mathbb{D}}|(f^{\prime}_n\circ\phi)(z)|^2dA(z)\\&\leq
    M\epsilon^2\|C_\phi\|^2_{A^2}\|f^{\prime}_n\|^2_{A^2},
 \end{align*}

 \begin{align*}
     \int_{\mathbb{D}\setminus S_1}|\psi^{\prime\prime}(z)|^2|(f_n\circ\phi)(z)|^2(1-|z|^2)^2dA(z)&\leq\epsilon^2\int_{\mathbb{D}\setminus S_1}|(f_n\circ\phi)(z)|^2dA(z)\hspace{.4cm}\\&\leq
     \epsilon^2\int_{\mathbb{D}}|(f_n\circ\phi)(z)|^2dA(z)\\&\leq
     \epsilon^2\|C_\phi\|^2_{A^2}\|f_n\|^2_{A^2}
 \end{align*}
 and
 \begin{align*}
     \int_{\mathbb{D}\setminus S_1}|\psi(z)\phi^{\prime\prime}(z)(f^{\prime}_n\circ\phi)(z)|^2(1-|z|^2)^2dA(z)&\leq\epsilon^2 M\int_{\mathbb{D}}(1-|\phi(z)|^2)^2|(f^{\prime}_n\circ\phi)(z)|^2dA(z)\\&\leq
     \epsilon^2 M\|C_\phi\|^2_{A^2}\|f^{\prime}_n\|^2_{A^2}.
 \end{align*}
  Since $f_n\in\mathcal{D}$, it follows that $f^{\prime}_n\in A^2$.\\
  Therefore, by Lemma \ref{L2}, we have $$\frac{2}{3}\|f^{\prime}_n\|^2_{A^2}\leq |f^{\prime}_n(0)|^2+\int_{\mathbb{D}}|f^{\prime\prime}_n(z)|^2(1-|z|^2)^2dA(z)\leq 2\|f^{\prime}_n\|^2_{A^2}.$$
  Again,
\begin{eqnarray*}
   &&\int_{\mathbb{D}\setminus S_1}|\psi(z)|^2|\phi^{\prime}(z)|^4|(f^{\prime\prime}_n\circ\phi)(z)|^2(1-|z|^2)^2dA(z)\\&\leq&\epsilon^2 M^2\int_{\mathbb{D}} (1-|\phi(z)|^2)^2|\phi^{\prime}(z)|^2|(f^{\prime\prime}_n\circ\phi)(z)|^2dA(z).
    \end{eqnarray*}
    Let $w=\phi(z).$ Since $\phi$ is univalent, applying the change of variable formula from \cite[Theorem C, Page 163]{Measure_Halmos}, we obtain
    \begin{eqnarray*}
         &&  \int_{\mathbb{D}\setminus S_1}|\psi(z)|^2|\phi^{\prime}(z)|^4|(f^{\prime\prime}_n\circ\phi)(z)|^2(1-|z|^2)^2dA(z)\\ & \leq& \epsilon^2 M^2 \Big( \int_{\mathbb{D}}(1-|w|^2)^2|f^{\prime\prime}_n(w)|^2dA(w)+|f^{\prime}_n(0)|^2-|f^{\prime}_n(0)|^2\Big)\\&\leq&
    \epsilon^2 \left(\|f^{\prime}_n\|^2_{A^2}-|f^{\prime}_n(0)|^2\right).  
    \end{eqnarray*}
Therefore, $\|C_{\psi,\phi}(f_n)\|_{\mathcal{D}}\rightarrow0.$ 
By Lemma \ref{L1}, it follows that $C_{\psi,\phi}$ is compact in $\mathcal{D}$.
\end{proof}

\begin{example}
    Let $\phi(z)=\frac{z}{2-z}$ and $\psi(z)=(1-z)^2$. Then $\psi$ and $\phi$ satisfy the conditions of Theorem \ref{Th3}. So $C_{\psi,\phi}$ is compact in $\mathcal{D}.$
\end{example}

\section{Numerical range of weighted composition operators} \label{S2}

We begin this section with the following definition.
\begin{definition}
	For $r\in\mathbb{N}, r\geq 2$ and $r>j\geq 0,$ the set $P_j$ of $\mathcal{D}$ is defined by
	$$\\P_j=\left\{ f  : f(z)=z^jg(z^r), g \in \mathcal{D} \right\}.$$
\end{definition}
The set $P_j$ is clearly a closed subspace of $\mathcal{D}$. Since these subspaces are mutually orthogonal for  $r>j\geq 0$ and their sum covers $\mathcal{D}$, we arrive at the following lemma.
\begin{lemma}\label{L11}
	For each integer $r\geq 2$,
	$$\mathcal{D}=P_0 \oplus P_1 \oplus \ldots \oplus P_{r-1}.$$
    \end{lemma}
    Next, we find the numerical range of certain composition operators.
    \begin{theorem}\label{Th4}
	Let $r\geq2$ be an integer, $\phi(z)=\mu z$ with $\mu=e^{2\pi i/r}$ and $\psi(z)=g(z^r)$ for some $g \in M(\mathcal{D}) $ such that   $C_{\psi,\phi} \in \mathbb{B}\left(\mathcal{D}\right) $, then 
	$$W(C_{\psi,\phi})=\left(\mu^0W(M_{\psi}|_{P_0})\cup \mu^1W(M_{\psi}|_{P_1})\cup \ldots \cup \mu^{(r-1)}W(M_{\psi}|_{P_{(r-1)}})\right)^{\wedge}.$$
\end{theorem}

\begin{proof}
    
	Let $f\in P_j$ with  $r>j\geq 0$. Then  $f(\phi(z))=f(\mu z)=\mu^jf(z)$, which shows that $C_{\phi}(P_j)\subseteq P_j.$ Also $\psi(z)=g(z^r)$ for some  $g \in M(\mathcal{D}) $. For any $f\in P_j$, we can write $f(z)=z^jg_1(z^r)$, for some $g_1\in \mathcal{D}$. Since $g\in M(\mathcal{D})$, we have 
\begin{eqnarray*}
    M_{\psi}(f)(z)&=&{\psi}\cdot f\\&=&g(z^r)\cdot z^jg_1(z^r)\\&=&z^j(gg_1)(z^r),
\end{eqnarray*}
which is again an element of $P_j$. Hence, $M_{\psi}(P_j)\subseteq P_j$  and consequently $C_{\psi,\phi}(P_j) \subseteq P_j.$ From Lemma \ref{L11}, we know that 
$$\mathcal{D}=P_0 \oplus P_1 \oplus \ldots \oplus P_{r-1}.$$Therefore the operator $C_{\psi,\phi}$ decomposes as
	$$C_{\psi,\phi}=C_0\oplus C_1 \oplus \ldots \oplus C_{r-1},$$
	where $C_j=C_{\psi,\phi}|_{P_j}.$ If  $h \in P_j$ with $\|h\|_{\mathcal{D}}=1,$ we have
	$\langle C_jh,h \rangle_{\mathcal{D}}=\mu^j \langle \psi h,h \rangle_{\mathcal{D}}$. Thus   $W(C_j)=\mu^jW(M_{\psi}|_{P_j}).$
    Therefore,
	$$W(C_{\psi,\phi})=\left(\mu^0W(M_{\psi}|_{P_0})\cup \mu^1W(M_{\psi}|_{P_1})\cup \ldots \cup \mu^{(r-1)}W(M_{\psi}|_{P_(r-1)})\right)^{\wedge}.$$

\end{proof}
 First we study the numerical range of $C_{\psi,\phi}$ where $\phi(z)= w.$
\begin{lemma}\label{L3}
Let $\phi(z)= w $ with $|w|<1$, and let $\psi$ be a nonzero function such that  $C_{\psi,\phi}\in \mathbb{B}\left(\mathcal{D}\right) $, then
\begin{enumerate}[\upshape (i)]
\item  If $k_w\equiv c\psi\, (c\neq0)$, then $W(C_{\psi,\phi})=[0,\overline{c}\|\psi\|_{\mathcal{D}}^2]$.\\
\item If $k_w$ and $\psi$ are orthogonal to each other, then $W(C_{\psi,\phi})$ is the closed disk centered at origin with radius $\frac{1}{2}\|\psi\|_{\mathcal{D}}\|k_w\|_{\mathcal{D}}$.\\
\item In any other cases, $W(C_{\psi,\phi})$ is a closed elliptical disk whose foci are 0 and $\psi(w)$.
\end{enumerate}

\end{lemma}
\begin{proof}
    As $\phi\equiv w$, then for every $f\in\mathcal{D}$, $$C_{\psi,\phi}(f)=\psi f(w)=\langle f,k_w\rangle\psi.$$ 
    Hence, $C_{\psi,\phi}$ is a rank one operator, and the result follows from \cite[Proposition 2.5]{BS_IEOT_2002}.
\end{proof}
The next lemma follows from \cite[Theorem 2.6]{KEY_JME_2021}.
\begin{lemma}\label{L4}
    Suppose $C_{\psi,\phi}$ is bounded with   non-constant $\phi$ and nonzero $\psi$. If either $\psi(z_0)=0$ for some $z_0\in\mathbb{D}$ or if $\phi$ is not injective, then $0\in$ $int$ $W(C_{\psi,\phi})$.
\end{lemma}
Now, in the next two theorems, we determine the presence of zero in the numerical range of $C_{\psi,\phi}.$
\begin{theorem}\label{Th5}
	Let $C_{\psi,\phi} \in \mathbb{B}\left(\mathcal{D}\right) $ with  $\psi\neq 0$ and $\phi(0)=0$. If $\phi$ is not a dilation map, then the numerical range of the weighted composition operator $C_{\psi,\phi}$ contains 0 in its interior.
\end{theorem}

\begin{proof}
	If $\phi^{\prime}(0)=0$ then $\phi$ is not injective, and Lemma \ref{L4} gives the result. Now, suppose $\phi^{\prime}(0)=\mu \neq 0.$ Since  $\phi$ is not of the form $\phi(z)=tz$ where $t$ belongs to $\overline{\mathbb{D}}$, it  can be expressed as
	\begin{align*}
		\phi(z)=\mu z \left( 1+bz^s(1+g(z)) \right),
	\end{align*} 
	where $s\in{\mathbb{N}}$, $b \neq 0$ and $g \in H(\mathbb{D})$ with $g(0)=0.$ Therefore, for any $r\in\mathbb{N}$
	\begin{align}\label{Theo2e1}
		\phi^r(z)=\mu^r z^r+rb\mu^rz^{r+ s}+\mbox{ terms of $z$ with higher powers}.
	\end{align}

Let
$\psi=\sum_{k=0}^{\infty}\hat{\psi}_kz^k$. If $\hat{\psi}_0=0$, we get the result from Lemma \ref{L4}. If $\hat{\psi}_0 \neq 0$, consider the two dimensional subspace $M_r$=span$\{e_r(z)=\frac{z^r}{\sqrt{r+1}}, e_{r+s}(z)=\frac{z^{r+s}}{\sqrt{r+s+1}}$\}. Then compression of $C_{\psi,\phi}$ to $M_r$  has the matrix form
	\begin{align*} 
		\begin{pmatrix}
			\hat{\psi}_0\mu^r&0\\
			\sqrt{\frac{r+s+1}{r+1}}\mu^r(rb\hat{\psi}_0+\hat{\psi}_s) &\hat{\psi}_0\mu^{r+s}
		\end{pmatrix}=\mu^rE_r,
	\end{align*}
	where 
	\begin{align*} 
		E_r=\begin{pmatrix}
			\hat{\psi}_0&0\\
			\sqrt{\frac{r+s+1}{r+1}}(rb\hat{\psi}_0+\hat{\psi}_s) &\hat{\psi}_0\mu^{s}
		\end{pmatrix}.
	\end{align*}
	Since $W(\mu^rE_r) \subseteq W(C_{\psi,\phi})$ by ( \cite[Proposition 1.4]{wu_gau_book}), it is enough to show that $0$ is in the interior of $W(E_r)$ for some $r.$ The numerical range $W(E_r)$ is an ellipse with foci $\hat{\psi}_0$ and $\hat{\psi}_0\mu^{s}$, and its minor-axis having length equals $\sqrt{\frac{r+s+1}{r+1}}|rb\hat{\psi}_0+\hat{\psi}_s|,$ as described in \cite[Example 3]{GR_BOOK}. Moreover,
	$$\lim_{r \to \infty}\frac{r+s+1}{r+1}=1.$$
	Thus, by choosing $r$ large enough, the point $0$ lies in the interior of $W(E_r)$, as required. 
\end{proof}
\begin{remark}\label{Re1}
If $C_{\psi,\phi}\in \mathbb{B}\left(\mathcal{D}\right)$ and $\phi$ satisfies the conditions of Theorem \ref{Th1} and Theorem \ref{Th5}, then for any nonzero $\psi\in\mathcal{D}$, $C_{\psi,\phi}$ is compact and $0\in W(C_{\psi,\phi}).$ Therefore, by \cite[Theorem 1]{De Barra} the numerical range $W(C_{\psi,\phi})$ is a closed set.
\end{remark}
\begin{example}
    When $\phi(z)=\frac{z^2}{2}$ and $\psi$ is any nonzero polynomial, Remark \ref{Re1} implies that the numerical range $W(C_{\psi,\phi})$ is closed.
\end{example}
Next, we examine whether zero lies in the numerical range when  $\phi(z)=tz$, $-1\leq t\leq 0$ and $\psi$ is a non-constant function.
\begin{theorem}\label{Th6}
    Let $C_{\psi,\phi} \in \mathbb{B}\left(\mathcal{D}\right) $ with $\psi$ a non-constant analytic function. If $\phi(z)=tz$ for $-1\leq t\leq 0,$ then $0$ lies in the numerical range $W(C_{\psi,\phi}).$
\end{theorem}

\begin{proof}
	When $\psi$ is non-constant and $t=0$, the result follows directly from Lemma \ref{L3}. For $-1\leq t< 0,$ start with the case $\psi(0)=0$, where Lemma \ref{L4} already provides the conclusion.
	Now, assume that $-1\le t<0$ and $\psi(0)\neq 0$. We may set $\psi(0)=1$ without loss of generality, so $\psi(z)=1+\xi(z),$ where  $\xi\in H(\mathbb{D})$ with $\xi(0)=0.$ Then $C_{\xi,\phi}\in \mathbb{B}\left(\mathcal{D}\right)$ and for every $f\in\mathcal{D}$
	\begin{align*}
		\langle C_{\psi,\phi}f,f \rangle_{\mathcal{D}}=\langle C_{\phi}f,f \rangle_{\mathcal{D}}+\langle C_{\xi,\phi}f,f \rangle_{\mathcal{D}}.
	\end{align*}
 	 For $-1 \leq t <0,$ and $\xi(0)=0$ Lemma \ref{L4} ensures  that zero is an interior point of  $W(C_{\xi,\phi})$. Hence, there exists $h \in \mathcal{D}$ with $\|h\|=1$ such that $Im \langle C_{\xi,\phi}h,h\rangle_{\mathcal{D}}>0.$ Since $\langle C_{\phi}h,h \rangle_{\mathcal{D}}$ is real and  $Im \langle C_{\xi,\phi}h,h\rangle_{\mathcal{D}}>0,$ we obtain a point $r_1\in W(C_{\psi,\phi})$ in the upper half plane. Using the same argument, there exists a point $r_2\in W(C_{\psi,\phi})$ in the lower half plane. Furthermore, since $\langle C_{\psi,\phi}e_1,e_1 \rangle_{\mathcal{D}}=t$ and $\langle C_{\psi,\phi}e_0,e_0 \rangle_{\mathcal{D}}=1$ with  ${e_1}(z)=\frac{z}{\sqrt2}$ and $e_0(z)=1$, we conclude that $0$ lies in the interior of  $W(C_{\psi,\phi})$.
\end{proof}

\begin{remark}\label{Re2}
    If $C_{\psi,\phi}\in \mathbb{B}\left(\mathcal{D}\right)$ and $\phi,\psi$  satisfy the conditions of Theorem \ref{Th3} and Theorem \ref{Th6}, then $C_{\psi,\phi}$ is compact and $0\in W(C_{\psi,\phi}).$ Therefore, by \cite[Theorem 1]{De Barra} the numerical range $W(C_{\psi,\phi})$ is a closed set. 
\end{remark}
\begin{example}
 when $\phi(z)=-\frac{z}{2}$ and  $\psi$ is any non-constant polynomial, Remark \ref{Re2} implies that the numerical range $W(C_{\psi,\phi})$ is closed.
\end{example}
\section{Numerical Ranges Containing Circular or Elliptical Disks }\label{S3}
 In this section, we concentrate on weighted composition operators whose numerical range includes either a circular disk or an elliptical disk. For such operators, we additionally determine the radius of the circular disk and the length of the axes of the elliptical disk.
\begin{theorem}\label{Th7}
	Let	$C_{\psi,\phi} \in  \mathbb{B}\left(\mathcal{D}\right)$ where $\phi(z)=e^{2\pi i/r} z$ $(r\in\mathbb{N})$ and $\psi(z)=\sum_{k=0}^{\infty}\hat{\psi}_kz^k.$ Suppose there exists positive integers  $s_2>s_1$ such that $\hat{\psi}_{rs_1}\hat{\psi}_{rs_2}\hat{\psi}_{r(s_2-s_1)}=0$ but at least one of these three coefficients is non-zero. Then $W(C_{\psi,\phi})$ contains a circular disk with centered  at $\hat{\psi}_0$ with radius 
    $$\frac{1}{2}\sqrt{(rs_1+1)|\hat{\psi}_{rs_1}|^2+(rs_2+1)|\hat{\psi}_{rs_2}|^2+\frac{(rs_2+1)}{(rs_1+1)}|\hat\psi_
    {r(s_2-s_1)}|^2}.$$
	
\end{theorem}

\begin{proof}
	Let $M_1$=span\{$e_0$, $e_{rs_1}$, $e_{rs_2}$\}, where $e_0=1$, ${e_{rs_1}}(z)=\frac{z^{rs_1}}{\sqrt{rs_1+1}}$ and ${e_{rs_2}}(z)=\frac{z^{rs_2}}{\sqrt{rs_2+1}}$. For these basis vectors, the operator $C_{\psi,\phi}$ acts as 
	\begin{align*}
	C_{\psi,\phi}e_0(z)=\sum_{k=0}^{\infty}\hat{\psi}_kz^{k},
	\end{align*}
	\begin{align*}
		C_{\psi,\phi}e_{rs_1}(z)=\frac{1}{\sqrt{rs_1+1}}\sum_{k=0}^{\infty}\hat{\psi}_kz^{rs_1+k}
	\end{align*}
	and 
	\begin{align*}
		C_{\psi,\phi}e_{rs_2}(z)=\frac{1}{\sqrt{rs_2+1}}\sum_{k=0}^{\infty}\hat{\psi}_kz^{rs_2+k}.
	\end{align*}
	 Thus the compression of $C_{\psi,\phi}$ to the subspace $M_1$ is represented by  the matrix 
    \begin{align*} 
	\begin{pmatrix}
		\hat{\psi_{0}} & 0 & 0 \\
		\sqrt{rs_1+1}\hat{\psi}_{rs_1} & \hat{\psi}_0 & 0 \\
		\sqrt{rs_2+1}\hat{\psi}_{rs_2} &
		\frac{\sqrt{rs_2+1}}{\sqrt{rs_1+1}}\hat{\psi}_
        {r(s_2-s_1)}& \hat{\psi}_0
	\end{pmatrix}.
\end{align*}
	According to  \cite[Theorem 4.1]{KRS_LAA_1997} the numerical range of the compression of $C_{\psi,\phi}$  to the subspace $M_1$ is the disk centered at $\hat{\psi}_0$ with radius 
    $$\frac{1}{2}\sqrt{(rs_1+1)|\hat{\psi}_{rs_1}|^2+(rs_2+1)|\hat{\psi}_{rs_2}|^2+\frac{(rs_2+1)}{(rs_1+1)}|\hat\psi_
    {r(s_2-s_1)}|^2}.$$Since the numerical range of the compression is contained in  $W( C_{\psi,\phi})$, the required result follows.
\end{proof}

\begin{theorem}\label{Th8}
	Let	 $C_{\psi,\phi} \in  \mathbb{B}\left(\mathcal{D}\right)$ where $\phi(0)=0$ and the weight function  $\psi$ has a zero  of multiplicity $r>0$ at the origin. If $\hat{\psi}_r$ denotes the coefficient of $z^r$  in the Taylor expansion of $\psi,$ then $W(C_{\psi,\phi})$ contains the disk centered at origin with radius $\frac{r+1}{r+2}|\hat{\psi}_r|$.
\end{theorem}

\begin{proof}
	Let $f(z)=\frac{1}{\sqrt{r+2}}(\mu+z^r)$,where  $|\mu|=1$, then $f\in \mathcal{D}$ and $\|f\|_\mathcal{D}=1.$ Assume $\phi(z)=\sum_{k=1}^{\infty}\hat{\phi}_kz^k$ and $\psi(z)=\sum_{k=r}^{\infty}\hat{\psi}_kz^k$, then
	\begin{align*}
		f(\phi(z))&=\frac{1}{\sqrt{r+2}}
		\left( \mu + \left(\sum_{k=1}^{\infty}\hat{\phi}_kz^k\right)^r\right)\\
		&=\frac{1}{\sqrt{r+2}}
		\left( \mu + \hat{\phi}_1^rz^r+ \mbox{ terms of $z$ with higher powers} \right).
	\end{align*}
	Thus,
	\begin{align*}
		& \langle C_{\psi,\phi}f,f \rangle_\mathcal{D}\\
		&= \langle \psi(z)f(\phi(z)),f(z) \rangle_\mathcal{D} \\
		& = \frac{1} {r+2}\left \langle  \left( \sum_{k=r}^{\infty}\hat{\psi}_kz^k \right)\left( \mu + \hat{\phi}_1^rz^r+ \mbox{terms of $z$ with higher powers} \right),\mu+z^r\right\rangle_\mathcal{D} \\
		&=\frac{r+1}{r+2}\mu \hat{\psi}_r.
	\end{align*} 
Since $\mu$ can be any unimodular number, the numerical range  $W( C_{\psi,\phi})$ contains the disk centered at origin with radius $\frac{r+1}{r+2}|\hat{\psi}_r|$. 	
\end{proof}

\begin{theorem}\label{Th9}
 	Let	$C_{\psi,\phi} \in  \mathbb{B}\left(\mathcal{D}\right)$ where $\phi(z)=\mu z$ with $\mu \neq 0$ and $\psi(z)=\sum_{k=1}^{\infty}\hat{\psi}_kz^k.$ Then for every integer $r\geq 2$, the numerical range $W(C_{\psi,\phi})$ contains a disk, centered at origin where  radius is $\frac{\sqrt{r+1}}{2\sqrt{2}} |\hat{\psi}_{r-1}\mu|$.
\end{theorem}

\begin{proof}
	For $r \geq 2$, let $M_2$ be the subspace of $\mathcal{D}$ spanned by $e_1$ and $e_r$, where ${e_1}(z)=\frac{z}{\sqrt2}$ and ${e_r}(z)=\frac{z^r}{\sqrt{r+1}}$. We then obtain
	\begin{align*}
		C_{\psi,\phi}e_1(z)=\frac{\mu}{ \sqrt{2}}\left( \sum_{k=1}^{\infty}\hat{\psi}_kz^{k+1} \right)
	\end{align*}
	and
	\begin{align*}
		C_{\psi,\phi}e_r(z)=\frac{\mu^r }{\sqrt{r+1}}\left( \sum_{k=1}^{\infty}\hat{\psi}_kz^{k+r} \right).
	\end{align*}
	Consequently, the compression of $C_{\psi,\phi}$ to $M_2$ is represented by the matrix 
	\begin{align*} 
		\begin{pmatrix}
			0&0\\
			\mu \frac{\sqrt{r+1}}{\sqrt{2}}\hat{\psi}_{r-1} &0
		\end{pmatrix}.
	\end{align*}
	Therefore, the numerical range of the compression of $C_{\psi,\phi}$ to $M_2$ is a closed disk centered at the origin and radius $\frac{\sqrt{r+1}}{2\sqrt{2}}|\hat{\psi}_{r-1}\mu|$ ( see \cite[Example 3]{GR_BOOK}). As a result, $W(C_{\psi,\phi})$ contains this disk.
\end{proof}

\begin{theorem}\label{Th10}
	Let	$C_{\psi,\phi} \in  \mathbb{B}\left(\mathcal{D}\right)$ with $\phi(z)=e^{2\pi i\theta}z$ and $\psi(z)=\sum_{k=0}^{\infty}\hat{\psi}_kz^k$ where $\theta$ is irrational. Then $W(C_{\psi,\phi})$ contains an elliptical disk where foci are $\hat{\psi}_0e^{2\pi ir\theta }$ and $\hat{\psi}_0e^{2\pi i(r+s)\theta }.$ The length of the major and the minor axes are 
	\begin{eqnarray*}
	    \sqrt{|\hat{\psi}_0|^2|e^{2\pi ir\theta }-e^{2\pi i(r+s)\theta }|^2+\frac{r+s+1}{r+1}|\hat{\psi}_{s}|^2}\,\,\,\mbox{and}\,\,\,\sqrt{\frac{r+s+1}{r+1}}|\hat{\psi}_{s}|
	\end{eqnarray*}
   respectively, where $r\geq 0$ and $s>0$ are integers.
   
\end{theorem}

\begin{proof}
	Let $M_s$ be the subspace of $\mathcal{D}$ spanned  by $e_r$ and $e_{r+s}$, where ${e_r}(z)=\frac{z^r}{\sqrt{r+1}}$ and ${e_{r+s}}(z)=\frac{z^{r+s}}{\sqrt{r+s+1}}$. Then we have
	\begin{align*}
		C_{\psi,\phi}e_r(z)=e^{2\pi ir\theta }\frac{1}{\sqrt{r+1}}\sum_{k=0}^{\infty}\hat{\psi}_kz^{k+r}
	\end{align*}
	and
	\begin{align*}
		C_{\psi,\phi}e_{r+s}(z)=e^{2\pi i(r+s)\theta }\frac{1}{\sqrt{r+s+1}}\sum_{k=0}^{\infty}\hat{\psi}_kz^{k+r+s}.
	\end{align*}
	Therefore, the compression of $C_{\psi,\phi}$ to $M_s$ has the matrix representation
	\begin{align*} 
		\begin{pmatrix}
			\hat{\psi}_0e^{2\pi ir\theta}&0\\
			e^{2\pi ir\theta }\sqrt{\frac{r+s+1}{r+1}}\hat{\psi}_{s}&\hat{\psi}_0e^{2\pi i(r+s)\theta}
		\end{pmatrix}.
	\end{align*}
	The  numerical range of the compression of $C_{\psi,\phi}$ to $M_s$ forms an ellipse with foci at $\hat{\psi}_0e^{2\pi ir\theta }$ and $\hat{\psi}_0e^{2\pi i(r+s)\theta },$  having a major axis of length 
	$$\sqrt{|\hat{\psi}_0|^2|e^{2\pi ir\theta }-e^{2\pi i(r+s)\theta }|^2+\frac{r+s+1}{r+1}|\hat{\psi}_{s}|^2}$$
	and a minor axis $\sqrt{\frac{r+s+1}{r+1}}|\hat{\psi}_{s}|$ (see \cite[Lemma 1.1-1]{GR_BOOK}). Since the numerical range of a compression is always contained in the numerical range of the operator, the desired result follows. 
\end{proof}

\textbf{Acknowledgments.}
 Mr. Subhadip Halder would like to thank UGC, Govt of India, for the financial support in the form of fellowship. Miss Sweta Mukherjee is supported by the State
Government Departmental Fellowship, Government of West Bengal.

{\textbf{Conflict of interest.}
 The authors declare no conflict of interest.
 \vspace{0.2cm}

 \vspace{0.2cm}
 {\textbf{Data availability.}
Not applicable.
 
\bibliographystyle{amsplain}

\end{document}